\documentclass[reqno,12pt]{amsart} 
\usepackage{setspace,graphicx,epstopdf,amsmath,amsfonts,amsgen,
amstext,amsthm,amsbsy,amsopn,amssymb,bbm,tikz,parskip,verbatim,mathrsfs,enumerate,xcolor,comment}  
\usepackage[utf8]{inputenc}   
\usepackage[top=2.5cm, bottom=4cm, left=2.5cm, right=2.5cm]{geometry}

\usepackage[pdfborder={0 0 0}]{hyperref}

\newtheorem{theorem}{Theorem}

\newtheorem{lemma}[theorem]{Lemma}
\theoremstyle{definition}
\newtheorem{definition}[theorem]{Definition}

\begin{document}

\title{Shifted critical threshold in the loop ${ \boldsymbol{O(n)}}$ model  at arbitrary small $\boldsymbol{n}$}
\author{Lorenzo Taggi}
\maketitle

\begin{abstract}
In the loop $O(n)$ model a collection of mutually-disjoint self-avoiding loops is drawn at random on a finite domain of a lattice with probability proportional to
 $$\tiny{\lambda^{ \# \mbox{ edges} } n^{ \# \mbox{ loops} },}$$
where $\lambda, n \in [0, \infty)$.
Let $\mu$ be the connective constant of the lattice and, for any $n \in [0, \infty)$, let  $\lambda_c(n)$ be the largest value of $\lambda$ such that the loop length admits uniformly bounded exponential moments.
It is not difficult to prove that  $\lambda_c(n) =1/\mu$   when $n=0$ (in this case the model corresponds  to the self-avoiding  walk) and that for any $n \geq 0$, $\lambda_c(n) \geq 1/\mu$. 
In this note we prove that, 
\begin{align*}
\lambda_c(n) & > 1/\mu  \, \, \, \, \, \, \, \, \, \, \, \mbox{  whenever $n >0$}, \\
\lambda_c(n) & \geq 1/\mu  \,  + \,  c_0 \, n  \, + \,
O(n^2),
\end{align*}
 on $\mathbb{Z}^d$, with $d \geq 2$, and on the hexagonal lattice, where $c_0>0$.
This means that, when $n$ is positive (even arbitrarily small), as a consequence of the mutual repulsion between the loops, a phase transition 
can only occur at a strictly larger critical threshold than in the self-avoiding  walk. 
\end{abstract}

\section{Introduction}
The loop $O(n)$ model is defined as follows. Consider an infinite undirected graph $\mathcal{G} =(V,E)$ of bounded degree.
For any finite sub-graph $G =(V_G, E_G) \subset \mathcal{G}$, let $\Omega_G$ be the set of spanning sub-graphs of $G$ such that every vertex  has degree either zero or two. It follows from this definition that every connected component of the graph $\kappa \in \Omega_G$ is either an isolated vertex or a loop.
For any $\kappa$, let $o_G(\kappa)$ be the total number of edges of $\kappa$ and let 
$L_G(\kappa)$ be the total number of loops of $\kappa$.
Let $n , \lambda \in [0, \infty)$ be two parameters. 
The  measure of the loop $O(n)$ model is a probability measure on $\Omega_G$ which assigns weights,
\begin{equation}\label{eq:def measure}
\mathbb{P}_{G,  \lambda, n}(\kappa) :=  \frac{\lambda^{o_G(\kappa)} \, n^{L_G(\kappa)}}{Z_{\lambda, n}(G)}, \, \, \, \, \kappa \in \Omega_G,
\end{equation}
where $Z_{\lambda, n}(G)$ is a normalizing constant, to which we will refer as \textit{partition function} (we adopt the convention that $0^0 = 1$).

The loop $O(n)$ model was introduced on the hexagonal lattice as a graphical representation of the spin $O(n)$ model \cite{Domany1}.
The central question concerning this model is describing the structure and the size of the loops in the limit of large graphs.
This model presents a mathematically interesting and rich behaviour, which depends on the value of the  parameters
and on the structure of the underlying graph. It can be viewed as a model for random polymers interacting with a random environment through a `rigid' potential.
The study of random polymers in random environment is of great physical and mathematical interest  (see for example \cite{denHollander} for a review).
Another reason to consider this model is that it interpolates between several paradigmatic statistical mechanics models, to which it reduces for specific values of $n$, and, thus, allows  to compare them. More precisely, the model reduces to  
self-avoiding walk when $n=0$, the Ising model when $n=1$, critical percolation  when $n=\lambda=1$, the dimer model when $n=1$ and $\lambda=\infty$,
 proper 4-coloring  when $n=2$ and $\lambda=\infty$, 
 integer-valued $(n = 2)$ and tree-valued (integer $n \geq 3$) Lipschitz functions and the hard hexagon model $(n = \infty)$
on the hexagonal lattice. We refer to \cite{Peled} for an extensive discussion. Some of these relations  are also valid on $\mathbb{Z}^d$ for a variant of this model where the loops are allowed to overlap and the number of overlaps receives a  weight
which depends on $n$ \cite{Chayes}.  
Furthermore, when $n=2$, the loop $O(n)$ model is related to nearest-neighbour random lattice permutations \cite{Betz, Grosskinsky}, whose study stems from physics, where they are related to the theory of Bose-Einstein condensation \cite{Feynmann},
and when $n=2$ and $\lambda =\infty$, it is related to the double-dimer model (the only difference is that in random permutations  and in the double-dimer model also `loops' of length two are  allowed).

We now briefly review the rigorous results on the loop $O(n)$ model.
It was proved in \cite{Copin2} that,  when $\mathcal{G}$ is the hexagonal lattice,  $\mathbb{H}$, and $n$ is large enough, 
 the loops  are \textit{exponentially small} for any value of $\lambda \in (0, \infty)$ and that at least two distinct regimes 
exist:  a disordered phase in which each vertex is unlikely to be surrounded by any loops (when $n\lambda^6$ is small),  and an ordered phase which is a small perturbation of one of the three ground states (when $n \lambda^6$ is large).
It was proved in \cite{Copin3}  that, when $\mathcal{G} = \mathbb{H}$,  $n \in [1,2]$,  and $\lambda = 1/  \sqrt{2 + \sqrt{2 -n}}$ (the so called \textit{Nienhuis' critical point}), the loop $O(n)$ model exhibits macroscopic loops. 
When $n=0$, the loop $O(n)$ model corresponds to the single non-interacting random self-avoiding polygon (a self-avoiding walk which returns to the starting vertex).
To see this formally, one could slightly modify the definition (\ref{eq:def measure}) and let  $L_G(\kappa)$ be the number of loops in $\kappa \setminus \mathcal{P}_o(\kappa)$, with $\mathcal{P}_o(\kappa)$ being the connected component of $\kappa$ containing the origin, $o$.
This way, when $n=0$, only the loop containing the origin can be observed and it gets a   weight proportional to $\lambda^{| \mathcal{P}_o|}$. It is well known that 
in this case the length of $\mathcal{P}_o$ admits uniformly bounded exponential moments when  $\lambda \in (0, 1/ \mu)$, with $\mu = \mu(\mathcal{G})$ being the so-called \textit{connective constant} of $\mathcal{G}$
(see (\ref{eq:connective constant}) for a definition).
The exact value of this constant  is known on the hexagonal lattice \cite{Copin4}, $\mu(\mathbb{H}) =  1/  \sqrt{2 + \sqrt{2}}$.
Moreover, it was proved in \cite{Copin1} (in a slightly different setting) that 
$\mathcal{P}_0$ is \textit{weakly space-filling} when $\lambda \in ( 1 / \mu, \infty)$.
A variant of this model, (\ref{eq:def measure}), where the loops are allowed to intersect and the number of overlaps is weighted throught some vertex-factors  which depend on $n$ has been considered  in \cite{Chayes}.
There, it was proved that, on the torus of $\mathbb{Z}^d$,
for any $d \geq 2$, if $n$ is a large enough integer, the loops are exponentially small for any value of $\lambda \in (0, \infty)$, and that, when $d=2$, for any positive integer $n$ a break of translational symmetry occurs at a non-trivial value of $\lambda$.
However, such results do not apply to the model under consideration in this paper, since they require that  the vertex-factors are  bounded from below and from above by positive constants uniformly in $n$.

Thus, only part of  the conjectured phase diagram of  the loop $O(n)$ model has been rigorously proved.
This note proves a new fact concerning the phase diagram and the loop structure of the loop $O(n)$ model in $\mathbb{H}$ and in $\mathbb{Z}^d$, $d \geq 2$.
Let $\lambda_c(n)$ be the supremum among all values of $\lambda$ such that the loops are exponentially small (see (\ref{eq:lambda_0}) for a formal definition).
In this paper we prove that, whenever $n >0$, $\lambda_c(n) > \lambda_c(0) = 1/ \mu(\mathcal{G})$. This means that, as a result of the mutual repulsion between the loops,
which is present only when $n> 0$, it is more difficult for the loops to be long and, thus, the regime of macroscopic loops (if it exists) can only occur above
a critical  threshold which is strictly larger than in the case of no interaction. 
This is in accordance with the conjecture which was formulated by Nienhuis \cite{Nienhuis0, Nienhuis1, Nienhuis2},
namely that on the hexagonal lattice the critical threshold is strictly increasing with $n$ when $n$ is in $[0,2]$ and, more precisely,
it equals $ 1/  \sqrt{2 + \sqrt{2 -n}}$. 
A similar fact was proved in \cite{Betz}, where it was proved that the critical threshold of random lattice permutations is strictly larger than $1/\mu(\mathcal{G})$,
but the proof presented there is not valid for the model under consideration in this paper, since it essentially requires the existence of `loops' of length two.
Moreover, we provide a bound on the speed of convergence of $\lambda_c(n)$ to $1/\mu$ as $n$ goes to zero,  
$\lambda_c(n) \geq 1/\mu + c_0 n + O(n^2)$, where $c_0 >0$, corroborating another qualitative feature of the predicted phase diagram.

For any $\kappa \in \Omega_G$, and $x \in V_G$, let $\mathcal{P}_x(\kappa)$ be the subgraph of $\kappa$ 
corresponding to the connected component containing $x$.
 Let  $| \mathcal{P}_x(\kappa)| $ be the number of edges of $\mathcal{P}_x(\kappa)$. 
If no edge of $\kappa$ has $x$ as end-point, then the graph $\mathcal{P}_x(\kappa)$ contains only the vertex 
 $x$  and $| \mathcal{P}_x(\kappa)| = 0$. 
We will not deal with arbitrary graphs $G \subset \mathcal{G}$, but with \textit{domains}. A  graph  $G = (V_G, E_G) \subset \mathcal{G} = (V,E)$ is a said to be a {domain} if  its  edge set is $E_G = \{ \{x,y\} \in E \, : \, x,y \in V_G \}$.  
For any  $\delta >0$, $n \in (0, \infty)$ and $\lambda \in [0, \infty)$, define
$$
\mathcal{L}(\delta, \lambda, n) := \sup\limits_{ \substack{ G \subset \mathcal{G} :  \\ G  { \mbox{ \tiny  finite domain } } } }
\sup\limits_{ x \in V(G) } \, \, 
\mathbb{E}_{G, \lambda, n} \big ( e^{ \delta | \mathcal{P}_x|}  \big ),
$$
where $\mathbb{E}_{G, \lambda, n}$ denotes the expectation with respect to $\mathbb{P}_{G, \lambda, n}$.
If for some $\delta>0$ the previous quantity is finite,  the loop length admits uniformly bounded exponential moments.
For any $n \in (0, \infty)$, we define the critical threshold,
\begin{equation}\label{eq:lambda_0}
\lambda_c(n) : = \sup  \Big \{ \lambda \in [0, \infty) \,  : \,  \mathcal{L}(\delta, \lambda, n) < \infty \mbox{ for some positive $\delta$} \Big \}. 
\end{equation}
 \begin{theorem}\label{theo:1}
Let $\mathcal{G}$ be $\mathbb{Z}^d$, with $d \geq 2$, or the hexagonal lattice, $\mathbb{H}$, and let $\mu = \mu(\mathcal{G})$ be the connective constant.
We have that,
\begin{align}
\label{eq:strict monotonicity}
\lambda_c(n) & > 1/\mu, \, \, \, \,  \, \, \, \, \,
\, \, \, \,  \forall n \in (0, \infty),  \\
\label{eq:small n}
\lambda_c(n)  & \geq  1/\mu +  c_0 \,   n + O(n^2),
\end{align}
where $c_0 = c_0(\mathcal{G}) \in (0, \infty)$ is a constant which depends only on $\mathcal{G}$.
\end{theorem}

Our proof is very simple and uses two ingredients. The first ingredient is the celebrated Kesten's pattern theorem, Theorem \ref{theo:Kesten} below,
which is used to prove that the ``typical'' loop presents  a huge number of many little `open loops' (a self-avoiding walk with one missing edge to make it a closed  loop). The second ingredient is a multi-valued map principle to show that it is expensive for the system not to close these many 'open loops'.
This leads to the  upper bound  $\mathbb{P}_{G, \lambda, n} ( \mathcal{P}_x = \tilde{\mathcal{P}})  \leq \lambda^{| \tilde {\mathcal{P}}|} c^{ |\tilde {\mathcal{P}}|}$
for some $c = (\lambda, n) \in (0, 1)$, which holds uniformly in $\tilde {\mathcal{P}}$, in $G$ and $x \in V_G$. The enhancement $\lambda_c(n)> 1/\mu$ follows from the fact that $c < 1$.

Our result leads to the following natural questions.
This paper proves that $\lambda_c(n) > \lambda_c(0)$ when $n>0$. Is $\lambda_c(n)$ a strictly increasing function of $n$? The  critical threshold of the loop $O(n)$ model on the hexagonal lattice has been conjectured to satisfy such a strict monotonicity property and it seems likely that the same is true also on $\mathbb{Z}^d$, $d \geq 2$.
Furthermore, can one prove that $\lambda_c(n) < \infty$ on $\mathbb{Z}^d$, $d \geq 3$, for some values of $n \in (0, \infty)$? This should be the case, at least for small values of $n$.

This note is organized as follows. In Section \ref{sect:Kesten} we present all the definitions and state Kesten's pattern theorem. In Section \ref{sect:proof}
we present the proof of Theorem \ref{theo:1}.

\section{Kesten's pattern Theorem}\label{sect:Kesten}
In this section  we introduce the  definitions which are necessary  to present the proof of Theorem \ref{theo:1} and we state  Kesten's pattern theorem.
All definitions and statements refer to $\mathbb{Z}^d$, with $d \geq 2$. Their generalization  to the hexagonal lattice, $\mathbb{H}$, is simple.

A \textit{self-avoiding walk} $\omega$ on $\mathbb{Z}^d$ beginning at the site $x \in \mathbb{Z}^d$ is defined as a sequence of sites  $(\omega(0), \omega(1), \ldots \omega(N))$ with $\omega(0) = x$, satisfying 
$| \omega(j + 1) - \omega(j) |_2 = 1$, where $| \, \cdot \, |_2$ denotes the $L_2$ norm,  and $\omega(i) \neq \omega(j)$ for all $i \neq j$.
We write $| \omega | = N$ to denote the \textit{length} of $\omega$. 
We let $SAW_x(N)$ be the total number of self-avoiding walks of length $N$ beginning at the site $x \in \mathbb{Z}^d$.
The limit 
\begin{equation}\label{eq:connective constant}
\mu := \lim\limits_{N \rightarrow \infty}   \big (  \,  |SAW_x(N)| \,  \big )^{\frac{1}{N}},
\end{equation}
exists \cite{Hammersley}, it  is known as \textit{connective constant}, and it satisfies $\mu=\mu(\mathbb{Z}^d) \in [d, 2d-1]$.

A pattern is a short self-avoiding walk occurring in a longer self-avoiding walk.
\begin{definition}
\textit{A pattern $P = (p(0), \ldots, p(n))$ is said to occur at the $j$-th step of the self-avoiding walk $\omega = (\omega(0), \ldots, \omega(N))$ if there exists a vector
 $v \in \mathbb{Z}^d$ such that $\omega(j+k) = p(k) + v$ for every $k = 0, \ldots, n.$}
\end{definition}
Kesten's pattern theorem does not apply to general patterns, but to \textit{proper internal patterns}.
\begin{definition}
\textit{A pattern $P$ is a proper internal pattern if for every $k \in \mathbb{N}$ there exists a self-avoiding walk on which $P$ occurs at $k$ or more different steps.}
\end{definition}
We are ready to state  Kesten's pattern theorem, which was proved in \cite{Kesten} (see also \cite{Madras}[Chapter 7]).
For a pattern $P$, an integer $ N$, a vertex $x \in \mathbb{Z}^d$, and  a real number $w$,  let  $SAW_x[N, w, P] \subset SAW_x(N)$ be the set of $N$-steps self-avoiding walks   presenting the pattern $P$ at less than $w$ steps.

\begin{theorem}[Kesten, 1963]\label{theo:Kesten}
Recall that $\mu = \mu(\mathbb{Z}^d)$ is the connective constant. For any proper internal pattern $P$, there exists an $a >0$ small enough such that 
\begin{equation}
\limsup\limits_{N \rightarrow \infty} \big (   |SAW_x[N, aN, P]|     \big )^{\frac{1}{N}} < \mu.
\end{equation}
\end{theorem}
Before presenting the proof of the main theorem, we will provide a rigorous definition of self-avoiding polygon and state one important property.
For $N \geq 4$, an $N$-step \textit{self-avoiding polygon} $\mathcal{P}$ is an undirected graph $\mathcal{P} \subset G$ consisting of $N$ nearest-neighbour sites 
and edges connecting them
with the following property: there exists a corresponding $(N-1)$-step self-avoiding walk $\omega$ having $|\omega(N-1) - \omega(0)|_2=1$ such that 
the vertex set of $\mathcal{P}$ contains all the elements of $\omega$ and the edge set of $\mathcal{P}$ contains the edge joining $\omega(N-1)$ to $\omega(0)$ and the $N-1$ edges joining $\omega(i-1)$ to $\omega(i)$ ($i=1, \ldots, N-1$).
Let $SAP_x(N)$ be the set of $N$-step self-avoiding polygons $\mathcal{P}$ such that one vertex of $\mathcal{P}$ is $x$.
We also define the set $SAP_x(1)$, which includes only one graph, the (degenerate) 1-step self-avoiding polygon $\mathcal{P}$, which contains only the vertex $x$ and no edges, and $SAP_x(N)$ is empty for $N=2$ or $N$ odd.


Hammersley proved in \cite{Hammersley2} the remarkable fact that the connective constant of the self-avoiding polygons exists and is the same as the connective constant of self-avoiding walks,
\begin{equation}\label{eq:connective constant polygons}
\mu(\mathbb{Z}^d) = \lim\limits_{N \rightarrow \infty}   \big (  \,  |SAP_x(N)| \,  \big )^{\frac{1}{N}}.
\end{equation}
From the super-multiplicativity property of self-avoiding polygons it also follows that 
\begin{equation}\label{eq:number SAP}
|SAP_x(N) | \leq \frac{(d-1)}{d} \, N \mu^N.
\end{equation}
(see for example \cite{Madras}[Equations (3.2.1) and (3.2.5)]).

\section{Proof of Theorem \ref{theo:1}}\label{sect:proof}
Fix a dimension $d \geq 2$.
We want to assign an orientation to self-avoiding polygons
in order define pattern occurrences.
For any vertex $x \in \mathbb{Z}^d$, any integer $N > 1$,
and any self-avoiding polygon $\mathcal{P} \in SAP_x(N)$, one can identify precisely two $N-1$ steps self-avoiding walks,
$\omega^1 = (\omega^1(0), \ldots, \omega^1(N-1))$,  and
$\omega^2 = (\omega^2(0), \ldots, \omega^2(N-1)) \in SAW_x(N-1)$,
such that, for any $k \in \{1,2\}$ and $i \in [0, N-2]$, $\{ \omega^k(i), \omega^k(i+1) \}$ is an edge of $\mathcal{P}$
and $\{ \omega(N-1), \omega(0)\}$ is an edge of $\mathcal{P}$.
Since the map which assigns to any self-avoiding polygon 
$\mathcal{P} \in SAP_x(N)$ the corresponding pair 
of self-avoiding walks $\{\omega^1, \omega^2\}$ is a bijection, we can define a new bijection
$f : SAP_x(N) \mapsto SAW_x(N-1)$
which assigns to any self-avoiding polygon 
$\mathcal{P} \in SAP_x(N)$ a unique self-avoiding walk
$f(\mathcal{P}) \in \{\omega^1, \omega^2\}$ 
in some arbitrary manner (for example, $f$ might depend on some features $\mathcal{P})$.
The function $f$ is fixed in the whole proof and its definition will never be  made explicit.
We say that a  pattern $P$ occurs at the step $j \in [0, N-1]$ of a self-avoiding polygon $\mathcal{P} \in SAP_x(N)$ if it occurs at the step $j \in [0, N-1]$ of the self-avoiding walk $f(\mathcal{P}) \in SAW_x(N)$.
We let $SAP_x(N, w, P) \subset SAP_x(N)$ be the set of self-avoiding polygons of length $N$  such 
that  the pattern $P$ is present at less than $w$ steps.

Consider a finite sub-graph $G =(V_G,E_G) \subset \mathbb{Z}^d$.
Let also
\begin{equation}\label{eq:definition partition function}
Z_{\lambda, n}(G) = \sum\limits_{\kappa \in \Omega_G} \lambda^{ o_G(\kappa)} n^{L_G(\kappa)},
\end{equation}
be the partition function, which depends on the graph $G$.

We now define one specific pattern.
Let $P^{\prime}$ be the pattern corresponding to the sequence of vertices $(o, \boldsymbol{e}_2, \boldsymbol{e}_1+ \boldsymbol{e}_2, \boldsymbol{e}_1)$,
with $o \in \mathbb{Z}^d$ being the origin
and $\boldsymbol{e_i}$ the Cartesian unit vectors (see Figure \ref{Figure}). It is not difficult to see that such a  pattern is proper internal.
We start with an auxiliary lemma, which involves the self-avoiding polygons presenting such a pattern at many steps.
Given two graphs $G_1 = (V_{G_1}, E_{G_1})  \subset G_2 = ( V_{G_2}, E_{G_2})$, we let $G_2 \setminus G_1$ be the graph whose vertex set is
$V_{G_2} \setminus V_{G_1}$ and whose edge set is $\{  \{x,y\} \in E_{G_2} \, \, : \, \,   x,y \in V_{G_2} \setminus V_{G_1} \}$.
\begin{lemma}\label{lemma:multivalued}
For any $a \in (0,1)$ and $N \in \mathbb{N}$, let $G=(V_G,E_G) \subset \mathbb{Z}^d$ be an arbitrary finite domain, let $x \in V_G$ be an arbitrary vertex, let $\mathcal{P} \in SAP_x(N)$ be such that
$\mathcal{P} \subset G$ and  such that $\mathcal{P} \notin  SAP_x(N, aN, P^{\prime})$.
Then,
$$\frac{Z_{\lambda,n}( G \setminus \mathcal{P})}{Z_{\lambda,n}(G)} \leq \frac{1}{\big  (    1 + \lambda^4 \, n    \big )^{aN}}.$$
\end{lemma}
\begin{proof}
Given a self-avoiding polygon $\mathcal{P} \subset G$ (which was defined as a graph), we let $\mathcal{U}(\mathcal{P})$ be the graph whose vertex set is 
$V_{\mathcal{P}}$ and whose edge set is $\{  \{x,y\} \in E_G \, : \, x,y \in V_{\mathcal{P}} \}$. 
Note that $\mathcal{P}$ does not necessarily equal $\mathcal{U}(\mathcal{P})$, but it is always contained in $\mathcal{U}(\mathcal{P})$.
The following relation holds,
\begin{equation}\label{eq:relation partitions1}
Z_{\lambda, n}( G) \geq Z_{\lambda, n}( G \setminus \mathcal{P} ) Z_{\lambda, n}( \mathcal{U} (\mathcal{P} ) ) .
\end{equation}
Indeed, in the right-hand side we have the weight of configurations $\kappa \in \Omega_G$ such that no loop contains one vertex in $V_{\mathcal{P}}$ and one vertex
in $V_G \setminus V_{\mathcal{P}}$ at the same time, while in the left-hand side we have the weight of \textit{all} configurations  $\kappa \in \Omega_G$.

For a self-avoiding polygon $\mathcal{P} \in  SAP_x(N)$ satisfying the assumptions of the lemma, let
 $x_1$, $x_2$, \ldots $x_{aN}$ be the sequence of 
the first $aN$ sites of $f(\mathcal{P})$ where the pattern $P^{\prime}$ occurs, ordered in order of appearance  along $f(\mathcal{P})$, writing $aN$ in place of $\lceil a N \rceil$.
For any $i \in [1, aN]$, let now $Q_i$ be the (unique)  self-avoiding polygon of length four containing the vertices $\{x_i, x_i + \boldsymbol{e_2},  x_i + \boldsymbol{e_1} +  \boldsymbol{e_2},  x_i + \boldsymbol{e_1}\}$ and the edges connecting them (see Figure \ref{Figure}).
Let $\cup_{i=1}^{aN} Q_i$ be the graph corresponding to the union of the vertex sets and of the edge sets of the self-avoiding polygons $Q_i$, $i \in [1,aN]$.
Since $\cup_{i=1}^{aN} Q_i \subset \mathcal{U} (\mathcal{P} )$ (here we use the fact that $G$ is a domain), we deduce that,
\begin{equation}\label{eq:relation partitions2}
 Z_{\lambda, n}( \mathcal{U} (\mathcal{P} ) ) \geq  Z_{\lambda, n}( \cup_{i=1}^{aN} Q_i ). 
\end{equation}
\begin{figure}
    \centering
\includegraphics[scale=0.40]{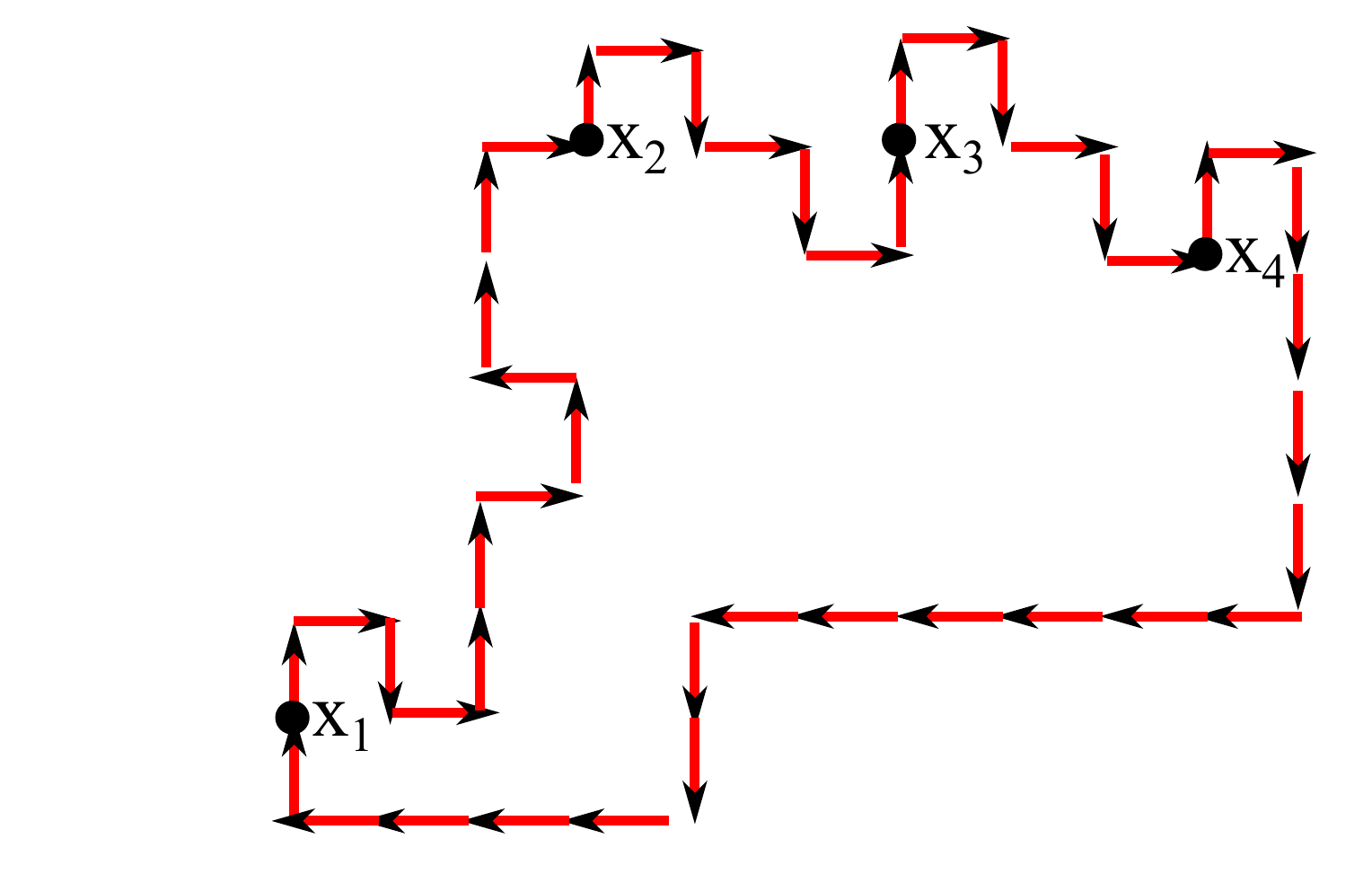}
\includegraphics[scale=0.40]{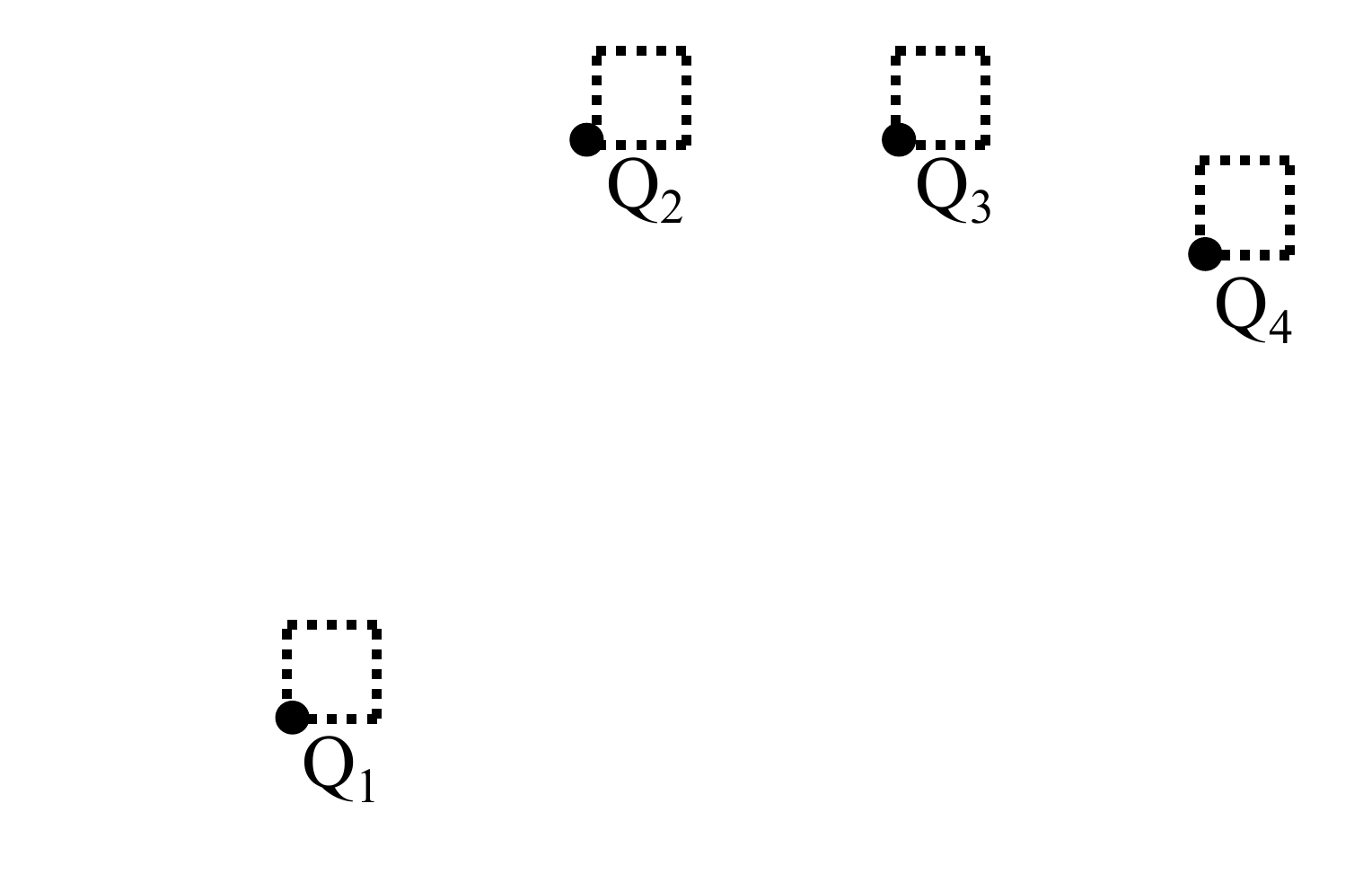}
\caption{\textit{{Left:}} A self-avoiding polygon ${\mathcal{P}}$ presenting the pattern $P^{\prime}$ at the vertices $x_i$. \textit{{Right:}} Self-avoiding polygons of length four, $Q_i$, for the self-avoiding polygon ${\mathcal{P}}$ represented on the left.  }\label{Figure}
\end{figure}
We deduce from (\ref{eq:relation partitions1}) and (\ref{eq:relation partitions2}) that,
\begin{equation}
\frac{Z_{\lambda,n}( G \setminus \mathcal{P})}{Z_{\lambda,n}(G)}
\leq  \frac{Z_{\lambda,n}( G \setminus \mathcal{P})}{Z_{\lambda, n}( G \setminus \mathcal{P} ) Z_{\lambda, n}( \mathcal{U} (\mathcal{P} ) )}  
\leq 
\frac{1}{  Z_{\lambda,n} \big ( \cup_{i=1}^{aN} Q_i \big ) }.
\end{equation}
We now claim that 
\begin{equation}\label{eq:bound partition}
Z_{\lambda,n} \big ( \cup_{i=1}^{aN} Q_i \big ) = (1 + \lambda^4 \, n)^{aN},
\end{equation}
which concludes the proof of the lemma when replaced in the previous expression.

Thus, for a subset $B \subset \{1, 2, \ldots aN\}$ (which might be $B = \emptyset$), let $\kappa_B \in \Omega_{ \cup_{i=1}^{aN} Q_i }$
be the configuration such that, for all $i \in B$,  ${\mathcal{P}}_{x_i} = Q_i$,
and for all $i \in  \{1, 2, \ldots aN\} \setminus B$, ${\mathcal{P}}_{x_i}$ is a degenerate self-avoiding polygon containing only the vertex $x_i$.
We have that, $L_G(\kappa_B)= |B|$ and that $o(\kappa_B) = \lambda^{4 |B|}.$
Thus,
$$
Z_{\lambda,n}( \cup_{i=1}^{aN} Q_i   ) = \sum\limits_{B \subset \{1, 2, \ldots aN\}} n^{|B|} \, \lambda^{4 |B|} = 
\sum\limits_{j=0}^{aN}  \binom{a N }{j}  \, n^{j} \, \lambda^{4j}
= \big  (    1 + \lambda^4 \, n    \big )^{aN}.
$$
This concludes the proof of (\ref{eq:bound partition}) and thus the proof of the lemma.
\end{proof}

We now present the proof of Theorem \ref{theo:1}. 
The starting point of the proof is the observation that, if $\mathcal{P} \in SAP_{x}[ N] $ with $N > 1$, then
\begin{equation}\label{eq:starting point}
\mathbb{P}_{G, \lambda, n} ( {\mathcal{P}}_x  = \mathcal{P})  = n \, \lambda^{|  \mathcal{P} |}  \,  \frac{ Z_{\lambda, n}(G \setminus \mathcal{P})   }{Z_{\lambda, n}(G)} 
\leq n  \lambda^{ | \mathcal{P}|}.
\end{equation}
We have that, for an arbitrary real $a \in (0,1)$, and $\ell \in \mathbb{N}$,
\begin{align}
\mathbb{P}_{G, \lambda, n} ( | {\mathcal{P}}_x |  > \ell) & =  \sum\limits_{N=\ell + 1}^{\infty} \, \, \sum\limits_{  \substack{ \mathcal{P} \in  SAP_x(N) : \\ \mathcal{P} \subset G } }
\mathbb{P}_{G, \lambda, n} ( {\mathcal{P}}_x  = \mathcal{P})  \\
& 
\label{eq:two terms}
=\sum\limits_{N=\ell + 1}^{\infty} \, \,  \Big ( \sum\limits_{  \substack{ \mathcal{P} \in  SAP_x(N, aN, P^{\prime}) : \\ \mathcal{P} \subset G } }
\mathbb{P}_{G, \lambda, n} ( {\mathcal{P}}_x  = \mathcal{P})  + 
\, \, \sum\limits_{  \substack{ \mathcal{P} \in  SAP_x(N) : \\  \mathcal{P} \notin  SAP_x(N, aN, P^{\prime}) , \mathcal{P} \subset G } }
\mathbb{P}_{G, \lambda, n} ( {\mathcal{P}}_x  = \mathcal{P}) \, \Big ) .
\end{align}
We will now provide an upper bound for the two terms above.
For the first term, we apply Kesten's pattern theorem, Theorem \ref{theo:Kesten}. Thus, fix $a^{\prime} > 0$ small enough such that
\begin{equation}\label{eq:choicea'}
 \mu^{\prime}:= \limsup_{N \rightarrow \infty} {|SAP_x[N, a^{\prime}N, P^{\prime}]|}^{\frac{1}{N}} \leq \limsup_{N \rightarrow \infty} {|SAW_x[N, a^{\prime}N, P^{\prime}]|}^{\frac{1}{N}}   < \mu.
 \end{equation}
Then, define $\lambda^{\prime}_1 := \frac{2}{\mu  + \mu^{\prime}}$, which satisfies $\lambda^{\prime}_1 > \frac{1}{\mu}$, and assume that $\lambda \in (0, \lambda^{\prime}_1)$.  We deduce from (\ref{eq:starting point}) and  (\ref{eq:choicea'}) that there exists a constant  $c_1 \in (0, \infty)$ such that, for any $\ell \in \mathbb{N}$,
\begin{multline}\label{eq:bound first term}
\sum\limits_{N=\ell + 1}^{\infty} \,  \sum\limits_{  \substack{ \mathcal{P} \in  SAP_x(N, a^{\prime}N, P^{\prime}) : \\ \mathcal{P} \subset G } }
\mathbb{P}_{G, \lambda, n} ( {\mathcal{P}}_x  = \mathcal{P})  \\ 
\leq \ \, n \, \  \, \sum\limits_{ N=\ell + 1}^{\infty} \, \, |  SAP_x(N, a^{\prime}N, P^{\prime}) | \lambda^N   \\
\leq c_1 \, \sum\limits_{ N=\ell + 1}^{\infty} \, \,  \big ( \frac{\mu  + \mu^{\prime}}{2} \big )^N \lambda^N \leq   \,  \frac{c_1}{1 - \frac{\lambda}{\lambda^{\prime}_1}} \, {(\frac{\lambda}{\lambda^{\prime}_1})}^{(\ell+1)}.
\end{multline}

We now use the previous lemma to provide an upper bound for the second term in the right-hand side of 
(\ref{eq:two terms}).
From (\ref{eq:number SAP}),  (\ref{eq:starting point}) and Lemma \ref{lemma:multivalued}, we deduce that, if
\begin{equation}\label{eq:condition lambda}
\lambda < \frac{\big  (    1 + \lambda^4 \, n    \big )^{a^{\prime} N}}{\mu}
\end{equation} 
then there exists $c_2, c_3 \in (0, \infty)$, which depend only on $\lambda$ and $n$, such that,  for any $\ell \in \mathbb{N}$,
\begin{multline}\label{eq:bound second term}
\sum\limits_{N=\ell + 1}^{\infty} \,  \sum\limits_{  \substack{ \mathcal{P} \in  SAP_x(N) : \\  \mathcal{P} \notin  SAP_x(N, a^{\prime}N, P) , \mathcal{P} \subset G } }
\mathbb{P}_{G, \lambda, n} ( {\mathcal{P}}_x  = \mathcal{P}) \, \,   \\
=  \, n \,  \sum\limits_{N=\ell + 1}^{\infty} \,  \sum\limits_{  \substack{ \mathcal{P} \in  SAP_x(N) : \\  \mathcal{P} \notin  SAP_x(N, a^{\prime}N, P) , \mathcal{P} \subset G } }    \lambda^{|  \mathcal{P} |}   \frac{ Z_{\lambda, n}(G \setminus \mathcal{P})   }{Z_{\lambda, n}(G)}  \\
\leq \, n \,  \sum\limits_{N=\ell + 1}^{\infty} \, |SAP_x(N) | \lambda^N {\big (  \frac{1}{1 + \lambda^4 n} \big )}^{a^{\prime}  N}   \\
\leq  \, n \,    \frac{(d-1)}{d} \, \sum\limits_{N=\ell + 1}^{\infty} \,  \, N  \, \big ( \frac{\lambda \, \, \mu}{(1 + \lambda^4 n)^{a^{\prime}}} \big )^N = c_2 e^{-c_3 \ell}.
\end{multline}
Let $\lambda_1 = \lambda_1(n)$ be the solution of 
\begin{equation}\label{eq:lambdasecond}
\lambda \,  \mu \,  = \,  (1 + \lambda^4 n)^{a^{\prime}}
\end{equation}
and note that $\lambda_1(n) >  \frac{1}{\mu}$ for any $n > 0$ and that (\ref{eq:condition lambda}) and (\ref{eq:bound second term}) hold whenever 
$\lambda \in (0, \lambda_1)$.
Combining (\ref{eq:bound first term}) and (\ref{eq:bound second term}) in (\ref{eq:two terms}), we deduce that,
if  
\begin{equation}\label{eq:explicit}
\lambda <   \min \{  \lambda_1^{\prime},  \lambda_1 (n)\},
\end{equation}
we can find $\delta>0$ such that $\mathcal{L}(\delta,  \lambda, n) < \infty$.
Thus, we proved that $\lambda_c(n) \geq \min \{  \lambda_1^{\prime},  \lambda_1 (n)\} > 1/\mu$ and obtained (\ref{eq:strict monotonicity}).

We now prove (\ref{eq:small n}). 
Using the fact that, for any $n$ smaller than a positive value $n_0$,
$\min \{  \lambda_1^{\prime},  \lambda_1 (n)\} = \lambda_1 (n)$,
  using (\ref{eq:lambdasecond}) and performing a Taylor expansion, we obtain that,
  for any $n \in (0, n_0)$,
\begin{align*}
\lambda_c(n) - 1/\mu & \geq  \lambda_1(n) - 1/\mu \\
& = \frac{  (1 + \lambda^4_1(n)\, n )^{a^{\prime}} - 1}{\mu} \\
& = \frac{a^{\prime}}{\mu} \lambda_1^4(0)  \, \, n + O(n^2) \\
& = \frac{a^{\prime}}{\mu^5}\,  n + O(n^2).
\end{align*}
This leads to (\ref{eq:small n}) and concludes the proof.

\section{Acknowledgements}
The author thanks the two anonymous referees for  very useful comments and the German Research Foundation (grant number:  BE 5267/1, DFG)  for providing financial support.

\end{document}